\documentclass[12pt]{amsart}

\usepackage{color,fancyhdr}
\usepackage[all]{xy}

\usepackage{amsmath,amsthm,mathrsfs,amsfonts,
amssymb,times,latexsym,mathabx}

\usepackage[usenames,dvipsnames,svgnames,table]{xcolor}

\DeclareMathAlphabet{\curly}{U}{rsfs}{m}{n}  

\newtheorem*{theoremA}{Theorem A}
\newtheorem*{theoremB}{Theorem B}
\newtheorem{theorem}{Theorem}[section]
\newtheorem{lemma}{Lemma}[section]
\newtheorem{proposition}[lemma]{Proposition}
\theoremstyle{definition}

\theoremstyle{problem}
\newtheorem{problem}[theorem]{Problem}
\usepackage{xcolor}

\numberwithin{equation}{section}

\newcommand{\ZZ}{{\mathbb Z}}

\newcommand{\RR}{{\mathbb R}}

\makeatletter
\renewcommand{\pmod}[1]{\allowbreak\mkern7mu({\operator@font mod}\,\,#1)}
\makeatother

\newcommand{\be}{\begin{equation}}
\newcommand{\ee}{\end{equation}}

\renewcommand{\le}{\leqslant}
\renewcommand{\leq}{\leqslant}
\renewcommand{\ge}{\geqslant}
\renewcommand{\geq}{\geqslant}

\allowdisplaybreaks

\begin{document}

\title[Lagrange-like spectrum of perfect additive complements]
{Lagrange-like spectrum of perfect additive complements}

\author{Bal\'{a}zs B\'{a}r\'{a}ny$^1$}
\address{$^1$Department of Stochastics, Institute of Mathematics, Budapest University of Technology and Economics, M\H{u}egyetem rkp. 3., H-1111 Budapest, Hungary}
\email{balubs@math.bme.hu}

\author{Jin-Hui Fang$^2$}
\address{$^2$School of Mathematical Sciences, Nanjing Normal University, Nanjing 210023, PR China}
\email{fangjinhui1114@163.com}

\author{Csaba S\'{a}ndor$^3$}
\address{$^3$Department of Stochastics, Institute of Mathematics, Budapest University of Technology and Economics, M\H{u}egyetem rkp. 3., H-1111 Budapest, Hungary,}
\address{  Department of Computer Science and Information Theory, Budapest University of Technology and Economics, M\H{u}egyetem rkp. 3., H-1111 Budapest, Hungary, MTA-BME Lend\"ulet Arithmetic Combinatorics Research Group,
  ELKH, M\H{u}egyetem rkp. 3., H-1111 Budapest, Hungary .}
\email{csandor@math.bme.hu}

\thanks{3 Corresponding author.}
\thanks{J.H Fang is supported by the National Natural Science Foundation of China, Grant No. 12171246 and the Natural Science Foundation of Jiangsu Province, Grant No. BK20211282. B. B\'ar\'any acknowledges support from the grant NKFI FK134251 and K142169. Cs. S\'andor was supported by the NKFIH Grants No. K129335. B. B\'ar\'any and Cs. S\'andor was supported by the grant NKFI KKP144059 "Fractal geometry and applications".}
\keywords{additive complements, Lagrange spectrum, Lebesgue-measure}
\subjclass[2010]{Primary 11B34, Secondary 11J06}
\date{\today}%

\begin{abstract}

Two infinite sets $A$ and $B$ of non-negative integers are called \emph{perfect additive complements of non-negative integers}, if every non-negative integer can be uniquely expressed as the sum of elements from $A$ and $B$. In this paper, we define a Lagrange-like spectrum of the perfect additive complements ($\mathfrak{L} $ for short). As a main result, we obtain the smallest accumulation point of the set $\mathfrak{L} $ and prove that the set $\mathfrak{L} $ is closed. Other related results and problems are also contained.

\end{abstract}

\maketitle

\section{\bf Introduction}

Let $\mathbb{Z}$ be the set of integers. For nonempty sets $A,B$ of integers and an integer $n$, let $r_{A,B}(n)$ be the number of representations of $n$ as $a+b$, where $a\in A$ and $b\in B$. Two infinite sets $A$ and $B$ of non-negative integers are called \emph{perfect additive complements of non-negative integers}, if $r_{A,B}(n)=1$ for every non-negative integer $n$. For a non-negative integer $m$, denote by $\mathbb{Z}_{\ge m}$ the set of non-negative integers no less than $m$. For simplicity, we also denote $\ZZ_{\geq1}$ by $\ZZ^+$.\vskip1mm

In \cite{FangSandor}, Fang and S\'andor characterized \emph{the perfect additive complements $A,B$ of non-negative integers}.\vskip1mm

\begin{theoremA}\cite[Theorem 1.1]{FangSandor} The infinite sets $A,B$ of the non-negative integers form perfect additive complements if and only if
\begin{equation}\label{05101}
	\begin{split}
A&=\{\epsilon_0+\epsilon_2m_1m_2+\dots+\epsilon_{2k-2}m_1\dots m_{2k-2}+\dots:\epsilon_{2i}=0,1,\dots,m_{2i+1}-1\}\text{ and }\\
B&=\{\epsilon_1m_1+\epsilon_3m_1m_2m_3+\dots+\epsilon_{2k-1}m_1\dots m_{2k-1}+\dots:\epsilon_{2i-1}=0,1,\dots,m_{2i}-1\}
\end{split}
\end{equation}(or $A,B$ interchanged), where $m_i\in \mathbb{Z}_{\ge 2}$ for every $i\in\ZZ^+$.
\end{theoremA}

Let $S$ be a set of non-negative integers. Its counting function is defined by $S(x)=|S\cap [0,x]|$ for every $x\in \mathbb{Z}_{\ge 0}$. It is easy to see that if $A,B\subseteq \mathbb{Z}_{\ge 0}$ form perfect additive complements then $A(x)B(x)\ge x+1$ for every non-negative integer $x$. In particular, Fang and S\'andor showed that $\displaystyle \liminf_{x\to \infty}\frac{A(x)B(x)}{x}=1$, see \cite[Theorem~1.5]{FangSandor}.
Recently, Ma \cite{Ma} determined the $\displaystyle \limsup_{x\to \infty}\frac{A(x)B(x)}{x}$ for the sets $A$ and $B$ with the form \eqref{05101}.\vskip1mm

\begin{theoremB}\cite[Lemma 2.1]{Ma} Let $m_1,m_2,\dots$ be arbitrary integers no less than two. Then the sets $A$ and $B$ with the form \eqref{05101} are perfect additive complements of non-negative integers such that
\begin{eqnarray*}\label{05104}
\limsup_{x\rightarrow\infty}\frac{A(x)B(x)}{x}=\limsup_{k\rightarrow\infty}\frac{2}{1+D_k},
\end{eqnarray*}
where
\begin{equation}\label{eq:defDk}
D_k=\frac{1}{m_k}-\frac{1}{m_km_{k-1}}+\frac{1}{m_km_{k-1}m_{k-2}}-\dots
+(-1)^{k-1}\frac{1}{m_km_{k-1}\dots m_1}.
\end{equation}
\end{theoremB}

In this paper, we consider the properties of the set called {\it Lagrange spectrum of perfect additive components} $$\mathfrak{L}:=\left\{ \limsup_{k\to \infty }\frac{2}{1+D_k}: (m_i)\in \mathbb{Z}_{\ge 2}^{\mathbb{Z}^+}\right\},$$
where $D_k$ is defined in \eqref{eq:defDk}. In 2011, Chen and Fang \cite[Theorem~1]{ChenFang} obtained that
$$
\frac{2a+2}{a+2}\in \mathfrak{L}\text{ for any integer $a$ with $a\ge 2$}.
$$
In 2016, Liu and Fang \cite[Theorem~1.1]{LiuFang} extended this result by showing that
$$
\frac{2}{\frac{a-1}{ab-1}+1}\in \mathfrak{L}\text{ for any integers $a,b$ with $2\le a\le b$.}
$$
Recently, Ma \cite[Theorem~1.1 and Theorem~1.2]{Ma} proved that
$$
2\in \mathfrak{L}\text{ and }\left(\left(\frac{16}{9},2\right)\setminus \mathbb{Q}\right)\cap \mathfrak{L}\neq\emptyset,
$$
where $\mathbb{Q}$ denotes the set of rationals. Fang and S\'andor \cite[Theorem~1.5]{FangSandor} showed that
$$
\mathfrak{L}\subseteq\left[\frac32,2\right].
$$

The main theorem of this paper can be summarized as follows:
	\begin{theorem}\label{thm:main}\mbox{}
		\begin{enumerate}
			\item\label{it:closed} The set $\mathfrak{L}$ is closed.
			\item\label{it:accpoint} The set $[\frac32,\gamma_0)\cap\mathfrak{L}$ is countably infinite, and can be given explicitly, where $\gamma_0$ is the smallest accumulation point of $\mathfrak{L}$.
			\item\label{it:freiman} $\left[\frac74,2\right]\subseteq \mathfrak{L}$ but $\left[\frac{12}{7}-\delta,2\right]\not\subseteq \mathfrak{L}$ for any $\delta>0$.
			\item\label{it:zeromeas} The Lebesgue-measure of $[\frac32,\frac{17}{10}]\cap\mathfrak{L}$ is zero.
		\end{enumerate}
	\end{theorem}

We may write $[\frac32,\gamma_0)\cap \mathfrak{L}=\{\gamma_1,\gamma_2,\ldots\}$, where $\gamma_n$ is a monotone increasing sequence converging to $\gamma_0$, in particular,
$$
	\gamma_1=\frac32<\gamma_2=\frac85<\gamma_3=\frac{13}{8}<\gamma_4=\frac{109}{67}<\cdots<\gamma_0\approx1.62688284...
$$
All values of the sequence $\gamma_n$ can be determined explicitly, see Section~\ref{sec:prelim}.


It follows from Theorem~\ref{thm:main} that the set $\mathfrak{L}$ has some similar properties to the so-called Lagrange spectrum $LS$. Let $\alpha $ be a positive irrational number. Define $\displaystyle k(\alpha )=\limsup _{n,m\to \infty}\frac{1}{|n^2\alpha -nm|}$. Hurwitz \cite{Hurwitz} proved that $k(\alpha) \ge \sqrt{5}$ for every positive irrational number $\alpha $. The Lagrange spectrum
$$
LS:=\{ k(\alpha ): \alpha \hbox{ is a positive irrational number} \} .
$$
For results related to Lagrange spectrum, one may refer to \cite{CuFla}, \cite{Lima}, \cite{Freiman}, \cite{Markoff}, \cite{Moreira} and \cite{Parkkonen}.

It is well known that the Lagrange spectrum is closed, see \cite[Theorem~3.2]{CuFla}, furthermore, the least accumulation point of the Lagrange spectrum is $3$ and $l\in L, l<3$ if and only if $l=\sqrt{9-\frac{4}{z_n^2}}$, where $z_n$'s are the Markov integers, see \cite{Markoff}. The corresponding phenomena for the Lagrange-like spectrum of perfect additive complement follows by Theorem~\ref{thm:main}\eqref{it:closed} and Theorem~\ref{thm:main}\eqref{it:accpoint}.

Furthermore, Freiman's constant $F=\frac{2221564096+283748\sqrt{463}}{491993569}=4.527\dots $ is the name of the supremum of the set $\mathbb{R}\setminus LS$, that is $[F,\infty ) \subset LS$, but for any $\delta >0$, $[F-\delta ,\infty ) \not\subset LS$, see \cite{Freiman}. In point of view of Theorem~\ref{thm:main}\eqref{it:freiman}, the $\mathfrak{L}$ has also a Freiman-like constant, namely, there exists $\frac{12}{7}\leq c_0\leq\frac74$ such that
$$
c_0=\inf\{c\in\RR:[c,2]\subset \mathfrak{L}\}.
$$

\begin{problem}\label{2} Determine the exact value of $c_0$. Is it true that $c_0=7/4$? \end{problem}

There is another important similarity between the sets $LS$ and $\mathfrak{L}$, namely, both can be represented by using infinite iterated function systems (IFS). It is well known that every $\alpha $ can be written as a simple infinite continued fraction
$$
\alpha =m_0+\frac{1}{m_1+\frac{1}{m_2+\dots }}=:[m_0;m_1,m_2,\dots ],
$$
where $m_i \in \mathbb{Z}^+$. On the other hand if $m_i \in \mathbb{Z}^+$, then the above continued fraction defines a positive irrational number. Let us define a map $G_m(x)=\frac{1}{m+x}$ for every integer $m\in\ZZ^+$. Then
$$
[m_0;m_1,m_2,\dots ]=m_0+\lim_{k\to\infty}G_{m_1}\circ\cdots\circ G_{m_k}(0).
$$

If $\frac{1}{|n^2\alpha -nm|}>2$, then there exists a $k$ such that $\frac{m}{n}=\frac{p_k}{q_k}=[m_0;m_1,\dots,m_k]$, see for example \cite[Theorem~19]{Khin}. Hence $k(\alpha )=\limsup_{k\to \infty}\frac{1}{|p_k^2\alpha -p_kq_k|}$. In 1921, Perron \cite{Perron} proved
$$
\frac{1}{|p_k^2\alpha -p_kq_k|}=[0;m_k,m_{k-1},\dots ,m_1]+[m_{k+1};m_{k+2},\dots ].
$$
In particular,
$$
LS=\left\{\limsup_{k\to\infty}\left(G_{m_k}\circ\cdots\circ G_{m_1}(0)+m_{k+1}+\lim_{\ell\to\infty}G_{m_{k+2}}\circ\cdots\circ G_{m_\ell}(0)\right)|(m_i)\in\ZZ_{\geq1}^{\ZZ^+}\right\} .
$$

Now, let us define the maps $\widehat{G}_m(x)=\frac{2mx}{(m+2)x-2}$. By Theorem~B, we will show later that
\begin{equation}\label{eq:LSPACform}
\mathfrak{L}=\left\{\limsup_{k\to\infty}\widehat{G}_{m_k}\circ\cdots\circ\widehat{G}_{m_1}(2)|(m_i)\in\ZZ^{\ZZ^+}_{\geq2}\right\}.
\end{equation}

Moreira \cite{Moreira} showed that the map $\alpha\mapsto\dim_H\left([\sqrt{5},\alpha]\cap LS\right)=\overline{\dim}_B\left([\sqrt{5},\alpha]\cap LS\right)$ is monotone increasing and continuous on $[\sqrt{5},\infty)$, where $\dim_H$ denotes the Hausdorff dimension and $\overline{\dim}_B$ denotes the upper box-counting dimension. For the definition and basic properties of the Hausdorff- and box-counting dimension we refer to \cite{Falconer}.

\begin{problem}
	Is $\dim_H\left([\frac32,\alpha]\cap \mathfrak{L}\right)=\overline{\dim}_B\left([\frac32,\alpha]\cap \mathfrak{L}\right)$? Is the map $\alpha\mapsto\dim_H\left([\frac32,\alpha]\cap \mathfrak{L}\right)$ continuous?
\end{problem}

\section{\bf Preliminaries}\label{sec:prelim}

In this section, we summarize some basic facts in the theory of iterated function systems relevant for our later calculations. We say that a map $f\colon\RR\mapsto\RR$ is contracting if there exists a constant $0<c<1$ such that $|f(x)-f(y)|\leq c|x-y|$. By Banach's fixed point theorem, every contractive map $f$ has a unique fixed point $x=f(x)$. For a contractive map $f$, let us denote its unique fixed point by $\mathrm{Fix}(f)$.

Let $\Psi=\{f_1,\ldots,f_n\}$ be a finite collection of contractions, which we call {\it iterated function system (IFS)}. Hutchinson~\cite{Hut} showed that there exists a unique non-empty compact set $\Lambda$ such that
\begin{equation}\label{eq:lambda}
\Lambda=\bigcup_{i=1}^nf_i(\Lambda).
\end{equation}
The set $\Lambda$ is called the {\it attractor} of the IFS $\Psi$. In particular, if $B\subset\RR$ is a compact set such that $f_i(B)\subseteq B$ for every $i=1,\ldots,n$ then
\begin{equation}\label{eq:containment}
\Lambda=\bigcap_{k=1}^\infty\bigcup_{(i_1,\ldots,i_k)\in\{1,\ldots,n\}^k}f_{i_1}\circ\cdots\circ f_{i_k}(B)\subset B.
\end{equation}

Using \eqref{eq:containment}, one can prove the following simple observation.

\begin{lemma}\label{lem:zeromeasure}
	Let $\Psi=\{f_1,\ldots,f_n\}$ be a finite collection of contractions such that the contracting ratio of $f_i$ is $c_i$. If $\sum_{i=1}^nc_i<1$ then $\lambda(\Lambda)=0$, where $\lambda$ denotes the Lebesgue measure on the real line.
\end{lemma}

\begin{proof}
	Since $|f_i(x)-f_i(y)|\leq c_i|x-y|$ then $\lambda(f_{i_1}\circ\cdots\circ f_{i_k}(B))\leq c_{i_1}\cdots c_{i_k}\lambda(B)$ and so, by \eqref{eq:containment},
	$$
	\lambda(\Lambda)\leq\sum_{(i_1,\ldots,i_k)\in\{1,\ldots,n\}^k}\lambda(f_{i_1}\circ\cdots\circ f_{i_k}(B))=\left(\sum_{i=1}^nc_i\right)^k\lambda(B)\to0\text{ as }k\to\infty.
	$$
\end{proof}

Let us denote the distance between sets by $\mathrm{dist}$, that is, for $A,B\subseteq\RR$, let\linebreak $\mathrm{dist}(A,B)=\inf\{|x-y|\big|x\in A,\ y\in B\}$. With a slight abuse of notation, we write $\mathrm{dist}(x,A)=\mathrm{dist}(\{x\},A)$ for the distance of a point $x\in\RR$ and a set $A\subseteq\RR$.

\begin{lemma}\label{lem:convergence}
	Let $\Psi=\{f_1,\ldots,f_n\}$ be a finite collection of contractions such that the contracting ratio of every $f_i$ is at most $c\in(0,1)$. For every sequence $(i_1,i_2,\ldots)\in\{1,\ldots,n\}^{\ZZ^+}$ and every $x\in\RR$, $\liminf_{k\to\infty}f_{i_k}\circ\cdots\circ f_{i_1}(x)\in\Lambda$, where $\Lambda$ is the attractor of $\Psi$. In particular, for every open set $U\supset\Lambda$, for every $x\in\RR$ and for every sufficiently large $k$, $f_{i_k}\circ\cdots\circ f_{i_1}(x)\in U$.
\end{lemma}
\begin{proof}
By \eqref{eq:lambda}
\begin{align*}
\mathrm{dist}(f_{i_k}\circ\cdots\circ f_{i_1}(x),\Lambda)\leq\mathrm{dist}(f_{i_k}\circ\cdots\circ f_{i_1}(x),f_{i_k}\circ\cdots\circ f_{i_1}(\Lambda))\leq c^k\mathrm{dist}(x,\Lambda)\to0\text{ as $k\to\infty$,}
\end{align*}
where $0<c<1$ is chosen such that $|f_i(x)-f_i(y)|\leq c|x-y|$ for every $i=1,\ldots,n$ and $x,y\in\RR$. The claim then follows by the compactness of $\Lambda$.
\end{proof}

For every point $x\in\Lambda$, there exists an infinite sequence $\mathbf{i}=(i_1,i_2,\ldots)\in\{1,\ldots,n\}^{\ZZ^+}$ such that
$$
x=\lim_{k\to\infty}f_{i_1}\circ\cdots\circ f_{i_k}(0).
$$
Observe that the limit on the right-hand side exists since the maps $f_i$ are contractions. One can define a map $\Pi\colon\{1,\ldots,n\}^{\ZZ^+}\mapsto\Lambda$ by
$$
\Pi(\mathbf{i}):=\lim_{k\to\infty}f_{i_1}\circ\cdots\circ f_{i_k}(0)
$$
called the {\it natural projection}. Let $\sigma\colon\{1,\ldots,n\}^{\ZZ^+}\mapsto\{1,\ldots,n\}^{\ZZ^+}$ be the left-shift operator, that is,
$$
\sigma(i_1,i_2,\ldots)=(i_2,i_3,\ldots).
$$
Hence, by using the definition of the natural projection $\Pi$ it is easy to see that
$$
\Pi(\mathbf{i})=f_{i_1}(\Pi(\sigma\mathbf{i})).
$$

Now, let us define a specific family of contractive maps on $\RR$ as $T_m(x)=\frac{1-x}{m}$ for $m\in\ZZ_{\geq2}$. Then clearly for every $(m_i)\in\ZZ^{\ZZ^+}_{\geq2}$
\[
T_{m_k}\circ T_{m_{k-1}}\circ \cdots \circ T_{m_1}(0)=\frac{1}{m_k}-\frac{1}{m_km_{k-1}}+\frac{1}{m_km_{k-1}m_{k-2}}-\dots
+(-1)^{k-1}\frac{1}{m_km_{k-1}\dots m_1},
\]
which corresponds to  \eqref{eq:defDk}. Let
$$
\mathcal{L}=\left\{\liminf_{k\to\infty}T_{m_k}\circ\cdots \circ T_{m_1}(0)|(m_i)\in\ZZ^{\ZZ^+}_{\geq2}\right\}.
$$
Hence,
\begin{equation}\label{eq:LSPACfunction}
\mathfrak{L}=g(\mathcal{L}),
\end{equation}
where $g(x)=\frac{2}{1+x}$. Furthermore, $\widehat{G}_m(x)=g\circ T_m\circ g^{-1}$, thus, \eqref{eq:LSPACform} follows. Hence, our main theorem will follow from the following theorems.

\begin{theorem}\label{thm5}
	
	The set $\mathcal{L}$ is closed.
	
\end{theorem}

\begin{theorem}\label{thm2}
	
	$[0,\frac{1}{7}]\subset\mathcal{L} $.\vskip2mm
	
	
\end{theorem}

\begin{theorem}\label{thm4}
	
	$$\mathcal{L} \cap \bigcup_{n=0}^{\infty }\left(\frac{1}{6}+\frac{1}{93}\frac{1}{4^n},\frac{1}{6}+\frac{1}{84}\frac{1}{4^n}\right)=\emptyset .$$
	
\end{theorem}

Let $S\subset \mathbb{R}$ be a Lebesgue-measurable set. The Lebesgue-measure of $S$ will be denoted by $\lambda (S)$.

\begin{theorem}\label{thm3}
	
	$\lambda\left(\mathcal{L} \cap [\frac{3}{17},\frac{1}{3}]\right)=0$.
	
\end{theorem}

We introduce the following notations. Let $\mathbf{i}=(i_1,\ldots,i_n)\in\ZZ_{\geq2}^{n}$ be a finite word, then denote by $T_{\mathbf{i}}$ the map
$$
	T_{\mathbf{i}}=T_{i_1}\circ\cdots\circ T_{i_n}.
$$
Let $u,v$ be positive integers and $\underline{m}=(m_i)\in\ZZ^{\ZZ^+}_{\geq2}$. If $u\le v$ then let\linebreak $T_{m_{[u,v]}}(x)=(T_{m_u}\circ T_{m_{u+1}}\circ \dots \circ T_{m_v})(x)$, and if $u>v$ then let\linebreak $T_{m_{[u,v]}}(x)=(T_{m_u}\circ T_{m_{u-1}}\circ \dots \circ T_{m_v})(x)$. Finally, let us introduce the notation that for any sequence $\underline{m}\in\ZZ^{\ZZ^+}_{\geq2}$
\begin{equation}\label{eq:natproj}
\Pi(\underline{m})=\lim_{k\to\infty}T_{m_1}\circ\cdots T_{m_k}(0)=\lim_{k\to\infty}T_{m_{[1,k]}}(0)=\sum_{k=1}^\infty\frac{(-1)^{k-1}}{m_1\cdots m_k}.
\end{equation}

Let us define the sequences $M^{(n)}$ recursively. Let $M^{(1)}=2$, $M^{(2)}=3$ and let $M^{(n)}$ be the concatenation $M^{(n)}=M^{(n-1)}M^{(n-2)}M^{(n-2)}$  for $n\ge 3$, that is $M^{(3)}=(3,2,2)$, $M^{(4)}=(3,2,2,3,3)$ and so on. By the definition of $M^{(n)}$, it is easy to see that the length of the finite sequence $M^{(n)}$ is $l_n=\frac{2^n-(-1)^n}{3}$, and $M^{(n)}$ starts with $M^{(n-1)}$. Thus, it is possible to define the limiting infinite sequence $M=\lim_{n\to\infty}M^{(n)}$ as
\begin{eqnarray*}
M=(3,2,2,3,3,3,2,2,3,2,2,3,2,\ldots)=:(M_1,M_2,\ldots) \end{eqnarray*}
such that $(M_1,M_2,\dots,M_{l_n})=(M^{(n)})$ for every positive integer $n\ge 2$. Let
$$
\lambda_n=\mathrm{Fix}(T_{M^{(n)}})=T_{M^{(n)}}(0)\frac{M_1M_2\dots M_{l_n}}{M_1M_2\dots M_{l_n}+1},
$$
and
\begin{eqnarray*}
\lambda_0=\sum_{l=1}^{\infty}(-1)^{l-1}\frac{1}{M_1M_2\dots M_l}=0.2293\dots.
\end{eqnarray*}
We will prove that $\lambda _n$ is a strictly increasing sequence, $\lambda _n>\lambda _0$.

\begin{theorem}\label{thm1}
\begin{eqnarray*}
\lambda \in \mathcal{L},\hskip3mm \lambda>\lambda _0 \hskip3mm\mbox{if and only if} \hskip3mm  \lambda =\lambda _n \text{ for some $n\geq1$}.\end{eqnarray*}\end{theorem}\vskip2mm

\begin{proof}[Proof of Theorem~\ref{thm:main}]
	The first claim follows by \eqref{eq:LSPACfunction}, the fact the map $g(x)=\frac{2}{1+x}$ is continuous on $\RR^+$ and Theorem~\ref{thm5}. The second claim follows by Theorem~\ref{thm1} with the choices $\gamma_n=g(\lambda_n)$ for $n\geq0$. The third claim follows by the combination of Theorem~\ref{thm2} and Theorem~\ref{thm4} together with \eqref{eq:LSPACfunction}. Finally, the last claim follows by Theorem~\ref{thm3} and by using the continuity of the map $g$.
\end{proof}

\section{\bf Closedness of the spectrum}

\begin{proof}[Proof of Theorem~\ref{thm5}]
	Let $\alpha _n\in \mathcal{L}$ be a sequence such that $\displaystyle \lim_{n\to \infty }\alpha _n= \alpha$. Hence, for every $n\geq1$ there exists $\underline{m}^{(n)}\in \mathbb{Z}_{\ge 2}^{\mathbb{Z}^+}$, $\underline{m}^{(n)}=( m_1^{(n)},m_2^{(n)},\dots )$ such that $\displaystyle \liminf _{k\to \infty }T_{m_{[k,1]}^{(n)}}(0)=\alpha _n$. Let $\varepsilon _n=|\alpha -\alpha _n|$. Without loss of generality we may assume that $\varepsilon _n \searrow 0$.
	
	Let $l_1=0$ and let us choose $k_1$ such that $|T_{m_{[k_1,1]}^{(1)}}(0)-\alpha _1|<\varepsilon _1$.
	
	For $n\ge 2$, let $0<l_n<k_n$ be such that $|T_{m_{[l_n,1]}^{(n)}}(0)-\alpha _n|<\varepsilon _n$, $|T_{m_{[k_n,1]}^{(n)}}(0)-\alpha _n|<\varepsilon _n$, $T_{m_{[l,1]}^{(n)}}(0)>\alpha _n-\varepsilon _n$ for every $l\ge l_n$ and $\frac{5\varepsilon _{n-1}}{\varepsilon _n}<2^{k_n-l_n}$.	Let
	$$
	\underline{m}=(m_{l_1+1}^{(1)},\dots ,m_{k_1}^{(1)},m_{l_2+1}^{(2)},\dots ,m_{k_2}^{(2)},m_{l_3+1}^{(3)},\dots ,m_{k_3}^{(3)},\dots )=(m_1,m_2,\dots ).
	$$
	We will show that
	\begin{equation}\label{eq:toclosed}
	\alpha=\liminf_{k\to\infty}T_{m_{[k,1]}}(0).
	\end{equation}
	
	Let $a_N=\sum_{n=1}^N(k_n-l_n)$. To verify \eqref{eq:toclosed}, it is enough to prove that
	\begin{equation}\label{eq3}
		|T_{m_{[a_N,1]}}(0)-\alpha _N|<2\varepsilon _N\text{ for every $N\ge 1$}
	\end{equation}
	and
	\begin{equation}\label{eq4}
		T_{m_{[l,1]}}(0)>\alpha _N-3\varepsilon _N\text{ for every $a_N<l\le a_{N+1}$.}
	\end{equation}
	Indeed, in this case
	$$
	\lim_{N\to \infty}T_{m_{[a_N,1]}}(0)=\alpha\text{	and }\liminf_{l\to \infty}T_{m_{[l,1]}}(0)\ge\lim_{N\to\infty}\left(\alpha_N-3\varepsilon_N\right)=\alpha.
	$$
	
	To prove (\ref{eq3}) we argue by induction. Clearly,
	$$
	|T_{m_{[a_1,1]}}(0)-\alpha_1|=|T_{m_{[k_1,1]}^{(1)}}(0)-\alpha _1|<\varepsilon _1<2\varepsilon_1.
	$$
	Suppose that \eqref{eq3} holds for $N-1$. Then
	\begin{align*}
	|T_{m_{[a_N,1]}}(0)-\alpha _N|&\le |T_{m_{[a_N,1]}}(0)-T_{m_{[k_N,1]}^{(N)}}(0)|+|T_{m_{[k_N,1]}^{(N)}}(0)-\alpha _N|\\
	&=\frac{1}{m_{k_N}^{(N)}\dots m_{l_N+1}^{(N)}}|T_{m_{[a_{N-1},1]}}(0)-T_{m_{[l_N,1]}^{(N)}}(0)|+|T_{m_{[k_N,1]}^{(N)}}(0)-\alpha _N|\\
	&\le\frac{1}{2^{k_N-l_N}}\left( |T_{m_{[a_{N-1},1]}}(0)-\alpha_{N-1}|+|\alpha_{N-1}-\alpha|+|\alpha-\alpha_N|+|T_{m_{[l_N,1]}^{(N)}}(0) -\alpha_N|\right) +\varepsilon _N\\
	&<\frac{1}{2^{k_N-l_N}}(2\varepsilon_{N-1}+\varepsilon_{N-1}+\varepsilon_{N}+\varepsilon_{N})+\varepsilon_{N}<\frac{1}{2^{k_N-l_N}}5\varepsilon_{N-1}+\varepsilon_{N}<2\varepsilon_{N}.
	\end{align*}
	
	To prove (\ref{eq4}) we write
	\begin{align*}
	T_{m_{[l,1]}}(0)&=T_{m_{[l,a_N+1]}}\circ T_{m_{[a_N,1]}}(0)=T_{m_{[l-a_N+l_N,l_N+1]}^{(N)}}\circ T_{m_{[a_N,1]}}(0)\\
	&=T_{m_{[l-a_N+l_N,l_N+1]}^{(N)}}\circ T_{m_{[a_N,1]}}(0)-T_{m_{[l-a_N+l_N,l_N+1]}^{(N)}}\circ T_{m_{[l_N,1]}^{(N)}}(0)+T_{m_{[l-a_N+l_N,l_N+1]}^{(N)}}\circ T_{m_{[l_N,1]}^{(N)}}(0)),
\end{align*}
	where
	\begin{align*}
	&\left|T_{m_{[l-a_N+l_N,l_N+1]}^{(N)}}\circ T_{m_{[a_N,1]}}(0)-T_{m_{[l-a_N+l_N,l_N+1]}^{(N)}}\circ T_{m_{[l_N,1]}^{(N)}}(0)\right|\\
	&\qquad=\frac{1}{m_{l-a_N+l_N}^{(N)}\dots m_{l_N+1}^{(N)}}\left|T_{m_{[a_N,1]}}(0)-T_{m_{[l_N,1]}^{(N)}}(0)\right|\\
	&\qquad\le\frac{1}{2^{l-a_N}}(|T_{m_{[a_N,1]}}(0)-\alpha_N|+|\alpha_N-T_{m_{[l_N,1]}^{(N)}}(0)|)< \frac{1}{2}(2\varepsilon _N+\varepsilon _N)<2\varepsilon _N
	\end{align*}
	and
	$$
	T_{m_{[l-a_N+l_N,l_N+1]}^{(N)}}\circ T_{m_{[l_N,1]}^{(N)}}(0)=T_{m_{[l-a_N+l_N,1]}^{(N)}}(0)>\alpha _N-\varepsilon _N.
	$$
	Hence,
	$$
	T_{m_{[l,1]}}(0)>\alpha _N-\varepsilon_N-2\varepsilon_N=\alpha _N-3\varepsilon_N,
	$$
	which completes the proof.
\end{proof}

\section{\bf Estimates on the Freiman-like constant}

Let us consider the finite IFS $\Psi_4=\{T_2,T_3,T_4\}$. Let $I=\left[\frac17,\frac37\right]$. For $m\ge 2$, $T_m(I)=[\frac{4}{7m},\frac{6}{7m}]$, and \begin{eqnarray}\label{06243}
	I=\bigcup_{m=2}^{4}T_m(I).
\end{eqnarray}
Thus, by the uniqueness the attractor of $\Psi_4$ is $I=\left[\frac17,\frac37\right]$. By \eqref{06243} and direct calculation, we obtain the following statements.\vskip2mm

\begin{lemma}\label{lem:claim1} For every $z\in [\frac{1}{7},\frac{3}{7}]$ and $K\in \{2,3,4\}$, we have $\frac{1}{K}-\frac{1}{K}z\in [\frac{1}{7},\frac{3}{7}]$.\end{lemma}

\begin{lemma}\label{lem:claim2} For every $y\in [\frac{1}{7},\frac{3}{7}]$, there exist $K\in \{2,3,4\}$ and $z\in [\frac{1}{7},\frac{3}{7}]$ such that $y=\frac{1}{K}-\frac{1}{K}z$.\end{lemma}

In particular, it follows from Lemma~\ref{lem:claim2} that for every $y\in [\frac{1}{7},\frac{3}{7}]$, there exists an infinite sequence $(K_1,K_2,\ldots)\in\{2,3,4\}^{\ZZ^+}$ such that
$$
\Pi(K_1,K_2,\ldots)=y,
$$
where $\Pi$ is defined in \eqref{eq:natproj}.

\begin{proof}[Proof of Theorem \ref{thm2}]
	It infers from $\frac{4}{7m}\le \frac{6}{7(m+1)}$ for $m\ge 6$ that
	\begin{eqnarray}\label{06242} \bigcup_{m=6}^{\infty}T_m(I)=\left(0,\frac17\right].
	\end{eqnarray}\vskip2mm
	
	Suppose that $0<x<\frac{1}{7}$. By \eqref{06242}, we know that the real $x$ can be written as
	$x=\frac{1}{m}-\frac{1}{m}y$, where $m\ge 6$ and $y\in [\frac{1}{7},\frac{3}{7}]$. It follows that there exist sequences $K_1,K_2,\dots $, $K_k,\dots \in \{2,3,4\}$ and $z_1,z_2,\dots $, $z_k,\dots \in [\frac{1}{7},\frac{3}{7}]$ such that
	\begin{eqnarray*}
		y&=&\Pi(K_1,K_2,\ldots)=\sum_{k=1}^\infty\frac{(-1)^{k-1}}{K_1K_2\dots K_k}.
	\end{eqnarray*}
	Now, let
	\begin{eqnarray*}
	\underline{m}=(m_1,m_2,\dots)=(3,K_1,m,3,K_2,K_1,m,3,K_3,K_2,K_1,m,3,K_4,K_3,K_2,K_1,m,\dots).
	\end{eqnarray*}
	We will prove that
	\begin{eqnarray*}
	x=\liminf_{n\rightarrow\infty}T_{m_{[n,1]}}(0).
	\end{eqnarray*}

	First, observe that $T_{m_{[n,1]}}(0)\in\left[0,\frac12\right]$ for every $n\geq1$. Indeed, $T_m(\left[0,\frac12\right])\subset\left[0,\frac12\right]$ for every $m\geq2$.	On the other hand, since $T_3\left(\left[0,\frac12\right]\right)\subset\left[\frac{1}{6},\frac{1}{3}\right]\subset\left[\frac17,\frac37\right]$, we have that $T_{m_{[n,1]}}(0)\in\left[\frac17,\frac37\right]$ for every $n\geq1$ with $m_n=3$. Hence, it follows from Lemma~\ref{lem:claim1} and \eqref{06243} that if $m_n\neq m$, then $T_{m_{[n,1]}}(0)\in [\frac{1}{7},\frac{3}{7}]$.
	
	Simple calculations show that $m_k=m$ if and only if $k=\frac{n^2+5n}{2}$ for some $n\in\ZZ^+$. Furthermore, it is easy to see that
	\begin{eqnarray*}
		T_{m_{[\frac{n^2+5n}{2},1]}}(0)&=&\frac{1}{m}-\frac{1}{m}
		\left(\frac{1}{K_1}-\frac{1}{K_1K_2}+\dots+
		\frac{(-1)^{n-1}}{K_1K_2\dots K_n}+\frac{(-1)^{n}}{K_1\cdots K_n}T_{m_{[\frac{(n-1)^2+5(n-1)}{2}+1,1]}}(0)\right)\\
		&=&\frac{1}{m}-\frac{1}{m}y+O(\frac{1}{2^n})=x+O(\frac{1}{2^n}).
	\end{eqnarray*}
	
	This completes the proof of Theorem \ref{thm2}.
\end{proof}

Before we continue, we state a lemma on the position of the possible smallest accumulation points depending on the defining sequence.

\begin{lemma}\label{lemma0b}\mbox{}
	\begin{enumerate}
	\item\label{(vi)} Let $(m_i)\in\mathbb{Z}_{\geq2}^{\ZZ^+}$ be such that $m_i\in\{ 2,3,\ldots,K\}$ except at most finitely many $i$. Then
	$$\frac{1}{2K-1}\leq\liminf_{n\to \infty} T_{m_{[n,1]}}(0)\leq\limsup_{n\to \infty} T_{m_{[n,1]}}(0)\leq\frac{K-1}{2K-1}.$$
	
	\item\label{(vii)} Let $K\in \mathbb{Z}_{\ge 2}$ and $(m_i)\in\ZZ_{\geq2}^{\ZZ^+}$ such that $m_i\ge K$ for infinitely many integer $i$. Then $\displaystyle \liminf _{k\to \infty}T_{m_{[k,1]}}(0) \le \frac{1}{K+1}$.
\end{enumerate}
\end{lemma}

\begin{proof}
	To prove the first claim, it is enough to show that
	\begin{equation}\label{eq:enough}
		T_m\left(\left[\frac{1}{2K-1},\frac{K-1}{2K-1}\right]\right)\subseteq\left[\frac{1}{2K-1},\frac{K-1}{2K-1}\right]\text{ for every $2\leq m\leq K$.}
	\end{equation}
	Indeed,
	\[
	T_m\left(\left[\frac{1}{2K-1},\frac{K-1}{2K-1}\right]\right)=\left[\frac{K}{m(2K-1)},\frac{2K-2}{m(2K-1)}\right],
	\]
	where $\frac{1}{2K-1}\leq\frac{K}{m(2K-1)}$ if and only if $m\leq K$ and $\frac{2K-2}{m(2K-1)}\leq\frac{K-1}{2K-1}$ if and only if $m\geq2$.
	
	Then by \eqref{eq:containment}, \eqref{eq:enough} and Lemma~\ref{lem:convergence}, we get that $\liminf_{n\to\infty}T_{m_{[n,1]}}(0)\in\left[\frac{1}{2K-1},\frac{K-1}{2K-1}\right]$.\vskip2mm
		
	To show the last claim, let us argue by contradiction. If $\liminf_{k\to \infty}T_{m_{[k,1]}}(0)>\frac{1}{K+1}$, then there exists a $\delta>0$ such that $T_{m_{[n,1]}}(0)>\frac{1}{K+1}+\delta $ for every sufficiently large $n$. Then
	\begin{eqnarray*}
		T_{m_{[n,1]}}(0)=\frac{1}{m_n}-\frac{1}{m_n}T_{m_{[n-1,1]}}(0)< \frac{1}{m_n}-\frac{1}{m_n}(\frac{1}{K+1}+\delta).
	\end{eqnarray*} Hence, for every sufficiently large $n$ we have that $\frac{1}{K+1}+\delta <\frac{1}{m_n}-\frac{1}{m_n}(\frac{1}{K+1}+\delta)$, equivalently $\frac{1}{K+1}+\delta <\frac{1}{m_n+1}$ for sufficiently large $n$. Thus, $m_n\le K-1$ for every sufficiently large $n$, which is a contradiction.
\end{proof}

Finally, let us state a technical lemma.
\begin{lemma}\label{(v)}
 Let $(m_i)\in \{ 2,3,4\}^{\ZZ^+}$ such that if $(m_i,m_{i-1})=(4,2)$ then $m_{i-2}=2$.  Then\linebreak $\limsup_{n\to \infty} T_{m_{[n,1]}}(0)\le \frac{13}{31}$.
 \end{lemma}

\begin{proof}
	Observe that
	$$\min T_2\circ T_4\circ T_2\left(\left[\frac17,\frac37\right]\right)\geq\max\left(\bigcup_{\substack{i,j,k\in\{2,3,4\}^3\\(i,j,k)\neq(2,4,2)}}T_{i}\circ T_j\circ T_k\left(\left[\frac17,\frac37\right]\right)\right),
	$$
	where we recall that $[\frac17,\frac37]$ is the attractor of the IFS $\{T_2,T_3,T_4\}$. Thus, if $(m_i,m_{i-1},m_{i-2},m_{i-3})=(2,4,2,2)$ only for finitely many $i$ then
 $$
\limsup_{n\to\infty}T_{m_{[n,1]}}(0)\leq T_2\circ T_4\circ T_2\circ T_2\left(\frac17\right)=\frac{23}{56}<\frac{13}{31}.
 $$
 On the other hand, if $(m_i,m_{i-1},m_{i-2},m_{i-3})=(2,4,2,2)$ for infinitely many $i$ then
	$$
	\limsup_{n\to\infty}T_{m_{[n,1]}}(0)\leq T_2\circ T_4\circ T_2\circ T_2\left(\limsup_{n\to\infty}T_{m_{[n,1]}}(0)\right),
	$$
	which implies after some algebraic manipulations that $\limsup_{n\to\infty}T_{m_{[n,1]}}(0)\leq\frac{13}{31}$.
\end{proof}

\begin{proof}[Proof of Theorem \ref{thm4}]
	It follows from Lemma \ref{lemma0b}\eqref{(vii)} that if $m_n\ge 5$ for infinitely many $n$, then $$\liminf _{k\to \infty }T_{m_{[k,1]}}(0)\le \frac{1}{6}.$$ So we may assume without loss of generality that $m_n\in \{ 2,3,4\} $ for every $n\in \mathbb{Z}^+$.
	
	Direct computations show that
	$$
	\left(\frac{1}{6},\frac{1}{6}+\frac{1}{84}\right)\cap\bigcup_{\substack{(k,l)\in\{2,3,4\}^2\\(k,l)\neq(4,2)}}T_{k}\circ T_l\left(\left[\frac17,\frac37\right]\right)=\emptyset.
	$$
	Thus, by Lemma~\ref{lem:convergence} and the fact that $\left[\frac17,\frac37\right]$ is the attractor of the IFS $\{T_2,T_3,T_4\}$ we get that if
	$$
	\liminf _{k\to \infty }T_{m_{[k,1]}}(0)\in\mathcal{L}\cap \left(\frac{1}{6},\frac{1}{6}+\frac{1}{84}\right)
	$$
	with the sequence $(m_i)\in\{2,3,4\}^{\ZZ^+}$ then $(m_i,m_{i-1})=(4,2)$ for infinitely many $i\in\ZZ^+$, and in particular,
	\begin{equation}\label{eq:kell}
	\liminf _{k\to \infty }T_{m_{[k,1]}}(0)=\liminf _{\ell\to \infty }T_{m_{[k_{\ell},1]}}(0),
	\end{equation}
	where $k_1=\min\{i\geq2:(m_i,m_{i-1})=(4,2)\}$ and $k_\ell=\min\{i> k_{\ell-1}:(m_i,m_{i-1})=(4,2)\}$ for all $\ell\geq2$
	
	If $(m_{k_\ell},m_{{k_\ell}-1},m_{{k_\ell}-2})=(4,2,b)$ for some $b\in \{ 3,4\}$ for infinitely many $i$ then $\frac{1}{6}\ge T_{m_{[{k_\ell},{k_\ell}-2]}}(0)\ge T_{m_{[{k_\ell},1]}}(0)$. So we may assume that if
	\begin{equation}\label{eq:wemayassume}
	\text{$(m_i,m_{i-1})=(4,2)$ then $m_{i-2}=2$ for every $i$.}
	\end{equation}
	Furthermore, if for every $N$ there exist infinitely many $k\geq 2N+2$ such that $$(m_{k},m_{k-1},\dots ,m_{k-2N-1})=(4,2,2,\dots,2)$$ then since the map $T_{4}\circ \overbrace{T_2\circ\cdots\circ T_2}^{2N+1\text{-times}}$ is monotone increasing
	$$
	\liminf _{k\to \infty }T_{m_{[k,1]}}(0)\le\lim_{N\to\infty}T_{4}\circ \overbrace{T_2\circ\cdots\circ T_2}^{2N+1\text{-times}}(\frac12)=\frac{1}{6}.
	$$
	Hence, we may assume that there exists a non-negative integer $N_0$ such that
	\begin{equation}\label{eq:thisholds}
	\begin{split}
	&(m_{k_\ell},m_{{k_\ell}-1},\dots ,m_{{k_\ell}-2N_0-1})=(4,\overbrace{2,2,\dots,2}^{\text{$2N_0+1$-times}})\text{ for infinitely many $\ell$ but}\\
	&(m_{k_\ell},m_{{k_\ell}-1},\dots,m_{{k_\ell}-2N_0-3})=(4,\overbrace{2,2,\dots ,2}^{\text{$2N_0+3$-times}})\text{ only for a finite number of $\ell$}.
\end{split}
\end{equation}
	
Let us suppose that \eqref{eq:thisholds} holds. For short, let $p_{N_0}=(m_{k_\ell},m_{{k_\ell}-1},\dots ,m_{{k_\ell}-2N_0-1})$. Then by Lemma~\ref{(v)} and the fact that the maps $T_m$ are orientation reversing we get
	\begin{equation}\label{eq:limsup}
		\begin{split}
		\liminf_{\ell\to\infty}T_{m_{[k_\ell,1]}}(0)&\le T_{p_{N_0}}(\limsup_{k\to\infty}T_{m_{[k-2N_0-2,1]}}(0))\\
		&=\frac{1}{8}\frac{1-\frac{1}{4^{N_0+1}}}{1-\frac{1}{4}}+\frac{1}{8\cdot  4^{N_0}}\cdot\frac{13}{31}=\frac{1}{6}+\frac{1}{93}\frac{1}{4^{N_0}}.
		\end{split}
	\end{equation}
	
	If $(m_{k_\ell},m_{{k_\ell}-1},\dots,m_{{k_\ell}-2N_0-2})=(4,2,2,\dots ,2,a)$, where $a\in \{ 3,4\}$ then
	$$
	T_{m_{[{k_\ell},1]}}(0)\le T_{m_{[{k_\ell},{k_\ell}-2N_0-2]}}(0)=T_{p_{N_0}}\circ T_a(0)\leq T_{p_{N_0}}(\frac13)=\frac16.
$$
	 Hence, we may suppose that $(m_{{k_\ell}-2N_0-2},m_{{k_\ell}-2N_0-3})\in\{(2,3),(2,4)\}$. Thus, by Lemma \ref{lemma0b}\eqref{(vi)} and the fact that the maps $T_m$ are orientation reversing we get
	\begin{equation}\label{eq2}
		\begin{split}
	\liminf_{\ell\to\infty}T_{m_{[k_\ell,1]}}(0)&\ge T_{p_{N_0}}\circ T_2\circ T_{a}(\liminf_{k\to\infty}T_{m_{[k-2N_0-3,1]}}(0))\\
	&\geq T_{p_{N_0}}\circ T_2\circ T_{a}(\frac17)\geq T_{p_{N_0}}\circ T_2(\frac27)=\frac{1}{6}+\frac{1}{84}\frac{1}{4^{N_0+1}}.
\end{split}
\end{equation}
Finally, \eqref{eq2} with \eqref{eq:kell} and \eqref{eq:limsup} implies that
	$$
	 \frac{1}{6}+\frac{1}{84}\frac{1}{4^{N_0+1}}.\le \liminf_{k\to \infty}T_{m_{[k,1]}}(0)\le\frac{1}{6}+\frac{1}{93}\frac{1}{4^{N_0}},
	$$
	and so $\liminf_{k\to \infty}T^{(1)}_{m_{[k,1]}}(0) \notin \bigcup_{n=0}^{\infty }\left(\frac{1}{6}+\frac{1}{93}\frac{1}{4^n},\frac{1}{6}+\frac{1}{84}\frac{1}{4^n}\right)$.
\end{proof}

\section{\bf Computation of the Markov-like constant}

Throughout this section, we will consider the set $\mathcal{L}\cap\left(\frac15,\frac12\right)$. By Lemma~\ref{lemma0b}\eqref{(vii)}, for every $x\in\mathcal{L}\cap\left(\frac15,\frac12\right)$ if $x=\liminf_{n\to \infty}T_{m_{[n,1]}}(0)$ then $(m_i)\in\{2,3\}^{\ZZ^+}$. By \eqref{eq:enough},
$$
T_2\left(\left[\frac15,\frac25\right]\right)\cup T_3\left(\left[\frac15,\frac25\right]\right)\subseteq \left[\frac15,\frac25\right],
$$
and so by denoting the attractor of the IFS $\{T_2,T_3\}$ by $\Lambda\subset\left[\frac15,\frac25\right]$, we get by Lemma~\ref{lem:convergence} that
\begin{equation}\label{eq:containment2}
	\mathcal{L}\cap\left[\frac15,\frac25\right]\subset\Lambda.
\end{equation}

Furthermore, direct computations show that
\begin{equation}\label{eq:separation2}
	T_2\left(\left[\frac15,\frac25\right]\right)\cap T_3\left(\left[\frac15,\frac25\right]\right)=\emptyset.
\end{equation}
Hence, by choosing $\delta=\mathrm{dist}\left(T_2\left(\left[\frac15,\frac25\right]\right),T_3\left(\left[\frac15,\frac25\right]\right)\right)/3>0$, we get
\begin{equation}\label{eq:separation}
T_2\left(J\right)\cap T_3\left(J\right)=\emptyset\text{ and }T_2\left(J\right)\cup T_3\left(J\right)\subseteq J,
\end{equation}
where $J=\left[\frac15-\delta,\frac25+\delta\right]$.

\begin{lemma}\label{lem:oddeven}\mbox{}
	\begin{enumerate}
		\item\label{it:odd} Let $(i_1,\ldots,i_{2k+1})\in\ZZ_{\geq2}^{\ZZ^+}$ for some non-negative integer $k$, and let $(m_i)\in\ZZ_{\geq2}^{\ZZ^+}$ be such that
	$(m_{j},\ldots,m_{j-2k})=(i_1,\ldots,i_{2k+1})$ for infinitely many $j$. Then
	$$
	\liminf_{n\to\infty}T_{m_{[n,1]}}(0)\leq\mathrm{Fix}(T_{i_1}\circ\cdots\circ T_{i_{2k+1}}).
	$$
	
	\item\label{it:even}  Let $x\in\mathcal{L}\cap\left(\frac15,\frac25\right)$ and let $(i_1,\ldots,i_{2k})\in\{2,3\}^{2k}$ be such that $x\in T_{i_1}\circ\cdots\circ T_{i_{2k}}\left(\left[\frac15,\frac25\right]\right)$. Then
	$$
	x\geq\mathrm{Fix}(T_{i_1}\circ\cdots\circ T_{i_{2k}}).
	$$
	\end{enumerate}
\end{lemma}

\begin{proof}
	Let us show the first claim. Let $j_\ell$ be the sequence such that $(m_{j_\ell},\ldots,m_{j_\ell-2k})=(i_1,\ldots,i_{2k+1})$. Since the maps $T_m$ are orientation reversing,
	\begin{align*}
		\liminf_{n\to\infty}T_{m_{[n,1]}}(0)&\leq\liminf_{\ell\to\infty}T_{m_{[j_\ell,1]}}(0)\\
		&=T_{i_1}\circ\cdots\circ T_{i_{2k+1}}(\limsup_{\ell\to\infty}T_{m_{[j_\ell-2k-1,1]}}(0))\\
		&\leq T_{i_1}\circ\cdots\circ T_{i_{2k+1}}(\liminf_{n\to\infty}T_{m_{[n,1]}}(0)).
	\end{align*}
	Let us denote the fixed point of $T_{i_1}\circ\cdots\circ T_{i_{2k+1}}$ by $x_0$ and, for short, let $x=\liminf_{n\to\infty}T_{m_{[n,1]}}(0)$. Since the maps $T_m$ are linear and contracting, we get
	\[
	x-x_0\leq T_{i_1}\circ\cdots\circ T_{i_{2k+1}}(x)-T_{i_1}\circ\cdots\circ T_{i_{2k+1}}(x_0)=\frac{-1}{i_1\cdots i_{2k+1}}(x-x_0), \]
	thus the claim follows.\vskip2mm
	
	Now we turn to the second claim. Let $x\in\mathcal{L}\cap\left(\frac15,\frac25\right)$ be such that $x\in T_{i_1}\circ\cdots\circ T_{i_{2k}}\left(\left[\frac15,\frac25\right]\right)$. Suppose that $x=\liminf_{n\to\infty}T_{m_{[n,1]}}(0)$. By \eqref{eq:separation}
	\begin{equation}\label{eq:sep2}
	\mathrm{dist}\left(T_{i_1}\circ\cdots\circ T_{i_{2k}}\left(J\right),\bigcup_{\substack{(j_1,\ldots,j_{2k})\in\{2,3\}^{2k}\\(j_1,\ldots,j_{2k})\neq(i_1,\ldots,i_{2k})}}T_{j_1}\circ\cdots\circ T_{j_{2k}}\left(J\right)\right)>0.
	\end{equation}
	By Lemma~\ref{lem:convergence}, for every sufficiently large $n$,  $T_{m_{[n,1]}}(0)\in J$. Then by \eqref{eq:separation}, for every sufficiently large $n$,
	$T_{m_{[n,1]}}(0)\in T_{m_{[n,n-2k+1]}}(J)$. Hence, by \eqref{eq:sep2}, $T_{m_{[n,1]}}(0)\in T_{i_1}\circ\cdots\circ T_{i_{2k}}\left(J\right)$ if and only if $(m_{n},\ldots,m_{n-2k+1})=(i_1,\ldots,i_{2k})$, and so by our assumption on $x\in\mathcal{L}\cap T_{i_1}\circ\cdots\circ T_{i_{2k}}\left(\left[\frac15,\frac25\right]\right)$
	$$
	\liminf_{n\to\infty}T_{m_{[n,1]}}(0)=\liminf_{\ell\to\infty}T_{m_{[n_\ell,1]}}(0),
	$$
	where $n_\ell$ is the sequence such that $(m_{n_\ell},\ldots,m_{n_\ell-2k+1})=(i_1,\ldots,i_{2k})$. Since $T_m$ is orientation reversing
	\begin{align*}
	\liminf_{n\to\infty}T_{m_{[n,1]}}(0)&=\liminf_{\ell\to\infty}T_{m_{[n_\ell,1]}}(0)=T_{i_1}\circ\cdots\circ T_{i_{2k}}\left(\liminf_{\ell\to\infty}T_{m_{[n_\ell-2k,1]}}(0)\right)\\
	&\geq T_{i_1}\circ\cdots\circ T_{i_{2k}}\left(\liminf_{n\to\infty}T_{m_{[n,1]}}(0)\right).
	\end{align*}
	Thus the statement follows similarly than the first claim.
\end{proof}

Let us recall the definition of the sequences $M^{(n)}$. Let $M^{(1)}=2$, $M^{(2)}=3$ and let $M^{(n)}$ be the concatenation
\begin{equation}\label{eq:defM}
M^{(n)}=M^{(n-1)}M^{(n-2)}M^{(n-2)}\text{ for $n\ge 3$.}
\end{equation}
By definition, the length $l_n$ of $M^{(n)}$ satisfies the equation $l_n=l_{n-1}+2l_{n-2}$ for every $n\geq3$ with $l_1=l_2=1$, which implies a standard calculation that $l_n=\frac{2^n-(-1)^n}{3}$.

We say for any two compact intervals $[a,b]$ and $[c,d]$ that $[a,b]\prec[c,d]$ if $b<c$.

For any two closed intervals $[a,b]$ and $[c,d]$ with $[a,b]\cap[c,d]=\emptyset$, let $\mathrm{mid}([a,b],[c,d])$ be the closure of the bounded component of $\mathbb{R}\setminus\left([a,b]\cup[c,d]\right)$.

\begin{lemma}\label{lem:propM} For every $n\geq3$, $T_{M^{(n)}}(J)\subset T_{M^{(n-1)}}(J)$. Furthermore,\linebreak $T_{M^{(n)}}(J)\cap T_{M^{(n-1)}M^{(n-1)}}(J)=\emptyset$, $T_{M^{(n)}}(J)\prec T_{M^{(n-1)}M^{(n-1)}}(J)$  and\linebreak $\Lambda\cap\mathrm{mid}(T_{M^{(n)}}(J),T_{M^{(n-1)}M^{(n-1)}}(J))=\emptyset$ for every $n\geq2$. That is, there is no element of $\Lambda$ in-between $T_{M^{(n)}}(J)$ and $T_{M^{(n-1)}M^{(n-1)}}(J)$ for every $n\geq2$.
\end{lemma}

\begin{proof}
By \eqref{eq:defM} and \eqref{eq:separation}, clearly $T_{M^{(n)}}(J)\subset T_{M^{(n-1)}}(J)$. We prove the second claim by induction. Clearly, by \eqref{eq:separation} $T_3(J)\cap T_{2,2}(J)=\emptyset$, furthermore, since
$$
\Lambda\cap\left(\left[\frac15,\frac25\right]\setminus\left(T_2(\left[\frac15,\frac25\right])\cup T_3(\left[\frac15,\frac25\right])\right)\right)=\emptyset
$$
and
$$
\left[\frac15,\frac25\right]\setminus\left(T_2(\left[\frac15,\frac25\right])\cup T_3(\left[\frac15,\frac25\right])\right)\supset\mathrm{mid}(T_{2,2}(J),T_3(J))
$$
the claim holds for $n=2$.

Let us suppose that the claim holds for $n$. Then
$$
T_{M^{(n+1)}}(J)\cap T_{M^{(n)}M^{(n)}}(J)=T_{M^{(n)}}\left(T_{M^{(n-1)}M^{(n-1)}}(J)\cap T_{M^{(n)}}(J)\right)=\emptyset,
$$
moreover, since $T_{M^{(n)}}$ is orientation reversing, $T_{M^{(n)}}(J)\prec T_{M^{(n-1)}M^{(n-1)}}(J)$ implies that
$$
T_{M^{(n+1)}}(J)=T_{M^{(n)}}\left(T_{M^{(n-1)}M^{(n-1)}}(J)\right)\prec T_{M^{(n)}}\left(T_{M^{(n)}}(J)\right).
$$
Observe that
$$
\mathrm{mid}(T_{M^{(n+1)}}(J),T_{M^{(n)}M^{(n)}}(J))=T_{M^{(n)}}\left(\mathrm{mid}(T_{M^{(n-1)}M^{(n-1)}}(J),T_{M^{(n)}}(J))\right)\subset T_{M^{(n)}}(J)
$$
and so by \eqref{eq:separation}
$$
\Lambda\cap\mathrm{mid}(T_{M^{(n+1)}}(J),T_{M^{(n)}M^{(n)}}(J))=T_{M^{(n)}}(\Lambda)\cap T_{M^{(n)}}\left(\mathrm{mid}(T_{M^{(n-1)}M^{(n-1)}}(J),T_{M^{(n)}}(J))\right)=\emptyset.
$$
\end{proof}

Let us recall the definition of the sequence $\lambda_n$ and $\lambda_0$. For every $n\geq1$, let $\lambda_n=\mathrm{Fix}(T_{M^{(n)}})$. Since $\mathrm{Fix}(T_{M^{(n)}})\in T_{M^{(n)}M^{(n)}}(J)\subset T_{M^{(n)}}(J)$, by Lemma~\ref{lem:propM} we get
$$
\lambda_{n+1}<\lambda_n.
$$
Thus, the sequence $\lambda_n$ is convergent. Let us denote the limit $\lim_{n\to\infty}\lambda_n$ by $\lambda_0$. Then by $T_{M^{(n)}}(J)\subset T_{M^{(n-1)}}(J)$, we get that
$$
\lambda_0=\Pi(\underline{M}),
$$
where $\underline{M}=\lim_{n\to\infty}M^{(n)}=(M_1,M_2,\ldots)$ is the limiting sequence defined such that\linebreak $(M_1,M_2,\dots,M_{l_n})=M^{(n)}$ for every positive integer $n\ge 2$. So
\begin{eqnarray*}
	\lambda_0=\sum_{l=1}^{\infty}(-1)^{l-1}\frac{1}{M_1M_2\dots M_l}=0.2293\dots.
\end{eqnarray*}

First, we show the following proposition:
\begin{proposition}\label{prop:lambdainL}
	 $\{\lambda_1,\lambda_2,\ldots\}\subset\mathcal{L}\cap\left(\frac15,\frac25\right)$. In particular, $\lambda_0\in\mathcal{L}$.
\end{proposition}
Before we turn to its proof, we require the following technical lemmas.

\begin{lemma}\label{lem:prefix}
	Let $\underline{a}=(a_1,\ldots,a_k)$ and $\underline{b}=(b_1,\ldots,b_n)$ be finite sequences formed by the integers $\{2,3\}$. Suppose that there exist a prefix $\underline{a}'=(a_1,\ldots,a_{k'})$ of $\underline{a}$ with $k'\leq k$ and a prefix $\underline{b}'=(b_1,\ldots,b_{n'})$ of $\underline{b}$ with $n'\leq n$ such that $T_{\underline{a}'}(J)\prec T_{\underline{b}'}(J)$. Then $T_{\underline{a}}(J)\prec T_{\underline{b}}(J)$.
\end{lemma}

\begin{proof}
	Observe that for every compact intervals $A,B$, if $C\subset A$ and $D\subset B$ are compact intervals then $A\prec B$ implies $C\prec D$. Thus, the claim follows by $T_{\underline{a}}(J)\subset T_{\underline{a}'}(J)$ and $T_{\underline{b}}(J)\subset T_{\underline{b}'}(J)$.
\end{proof}

For the finite sequence $M^{(n)}=(M_1^{(n)},\ldots,M_{l_n}^{(n)})$ and $1\leq\ell\leq l_n-1$, let $\sigma^{\ell}M^{(n)}$ be the $l_n$-length word such that
$$
\sigma^{\ell}M^{(n)}=(M_{\ell+1}^{(n)},\ldots,M_{l_n}^{(n)},M_{1}^{(n)},\ldots,M_{\ell}^{(n)}),
$$
with the convention that $\sigma^{l_n}M^{(n)}=M^{(n)}$. Thus, $\sigma^\ell M^{(n)}$ can be defined for every $\ell\in\ZZ^+$ in a natural, periodic way.

\begin{lemma}\label{lem:fixgood}
	For every $n\geq3$ and $1\leq\ell\leq l_n-1$, $T_{M^{(n)}}(J)\prec T_{\sigma^{\ell}M^{(n)}}(J)$.
\end{lemma}

\begin{proof}
	Simple calculations show that
	\begin{equation}\label{eq:firstorder}
	T_{3,2,2}(J)\prec T_{3,3,3}(J)\prec T_{3,3,2}(J)\prec T_{2,2,3}(J)\prec T_{2,3,3}(J)\prec T_{2,3,2}(J).
	\end{equation}
	Clearly for $M^{(3)}=(3,2,2)$, we have $\sigma^1M^{(3)}=(2,2,3)$ and $\sigma^2M^{(3)}=(2,3,2)$, and for $M^{(4)}=M^{(3)}M^{(2)}M^{(2)}=(3,2,2,3,3)$ we have
	$$
	\sigma^1M^{(4)}=(2,2,3,3,3),\ \sigma^2M^{(4)}=(2,3,3,3,2),\ \sigma^3M^{(4)}=(3,3,3,2,2), \ \sigma^4M^{(4)}=(3,3,2,2,3).
	$$
	Thus, by Lemma~\ref{lem:prefix} and \eqref{eq:firstorder}, the claim follows for $n=3$ and $n=4$.
	
	Let us prove the rest by induction. So suppose that the claim holds for $n\ge4$.
	
	First, assume that $1\leq \ell\leq l_{n-1}$. Then by
	\begin{equation}\label{eq:forM1}
	M^{(n+1)}=M^{(n-1)}M^{(n-2)}M^{(n-2)}M^{(n-1)}M^{(n-1)}
	\end{equation}
	and $M^{(n)}=M^{(n-1)}M^{(n-2)}M^{(n-2)}$, we get that $M_{k}^{(n+1)}=M^{(n-1)}_{k}=M^{(n)}_k$ for every $1\leq k\leq l_{n-1}$, and so $\sigma^{\ell}M^{(n)}$ is a prefix of $\sigma^{\ell}M^{(n+1)}$. Since $M^{(n)}$ is a prefix of $M^{(n+1)}$, we get by the induction hypothesis $T_{M^{(n)}}(J)\prec T_{\sigma^{\ell}M^{(n)}}(J)$ that $T_{M^{(n+1)}}(J)\prec T_{\sigma^{\ell}M^{(n+1)}}(J)$ by Lemma~\ref{lem:prefix}.
	
	Now, assume that $l_{n-1}<\ell< l_n$ but $\ell\neq l_{n-1}+l_{n-2}$. Then by
	\begin{equation}\label{eq:forM2}
	M^{(n+1)}=M^{(n-1)}M^{(n-2)}M^{(n-2)}M^{(n-2)}M^{(n-3)}M^{(n-3)}M^{(n-1)},
	\end{equation}
	we get that $\sigma^{\ell-l_{n-1}}M^{(n-2)}$ is a prefix of $\sigma^{\ell}M^{(n+1)}$. Hence, again by the fact that $M^{(n-2)}$ is a prefix of $M^{(n+1)}$ and the assumption that $T_{M^{(n-2)}}(J)\prec T_{\sigma^{k}M^{(n-2)}}(J)$ for every $k\notin\{l_{n-2},2l_{n-2},\ldots\}$, the claim follows by Lemma~\ref{lem:prefix}.
	
	If $\ell=l_{n-1}+l_{n-2}$ then by \eqref{eq:forM2}, we get that $M^{(n-2)}M^{(n-2)}$ is a prefix of $\sigma^\ell M^{(n+1)}$, meanwhile, if $\ell=l_{n-1}+2l_{n-2}=l_n$ then by \eqref{eq:forM1}, we get that $M^{(n-1)}M^{(n-1)}$ is a prefix of $\sigma^\ell M^{(n+1)}$, thus, the claim follows by Lemma~\ref{lem:propM}.
	
	Now, suppose that $l_n<\ell<l_n+l_{n-1}$ then by \eqref{eq:defM} we get that $\sigma^{\ell-l_n}M^{(n-1)}$ is the prefix of $\sigma^{\ell}M^{(n+1)}$, and since $M^{(n-1)}$ is a prefix of $M^{(n+1)}$, the claim follows by Lemma~\ref{lem:prefix} and the induction hypothesis.
	
	If $\ell=l_n+l_{n-1}$ then by \eqref{eq:defM} $M^{(n-1)}M^{(n-1)}$ is a prefix of $\sigma^\ell M^{(n+1)}$, thus, the claim again follows by Lemma~\ref{lem:propM}.
	
	Finally, if $l_n+l_{n-1}<\ell<l_n+2l_{n-1}=l_{n+1}$ then by \eqref{eq:defM}, $\sigma^{\ell-l_n-l_{n-1}}M^{(n-1)}$ is a prefix of $\sigma^\ell M^{(n+1)}$, so the claim follows again by the induction hypothesis and Lemma~\ref{lem:prefix}.
\end{proof}

\begin{proof}[Proof of Proposition~\ref{prop:lambdainL}]
	For $n=1$ and $n=2$, let
	\[
	\underline{m}^{(1)}=(2,2,\ldots)\text{ and }\underline{m}^{(2)}=(3,3,\ldots).
	\]
	Since the maps $T_m$ are contractions, we get
	$$
	\liminf_{k\to\infty}T_{m_{[k,1]}^{(n)}}(0)=\lim_{k\to\infty}T_{m_{[k,1]}^{(n)}}(0)=\lambda_n,
	$$
	and so, $\{\lambda_1,\lambda_2\}\subset\mathcal{L}$.
	
	For every integer $n\geq3$, let us define the following sequence:
	\[
	\underline{m}^{(n)}=(M_{l_n}^{(n)},\ldots,M_1^{(n)},M_{l_n}^{(n)},\ldots,M_1^{(n)},\ldots).
	\]
	Then by Lemma~\ref{lem:convergence}, $T_{m_{[k,1]}^{(n)}}(0)\in\bigcup_{\ell=0}^{l_n-1}T_{\sigma^\ell M^{(n)}}(J)$ for every sufficiently large $k$. Since the maps $T_m$ are contractions, we get
	$$
	\lim_{k\to\infty}T_{m_{[kl_n,1]}^{(n)}}(0)=\lambda_n\in T_{M^{(n)}}(J),
	$$
	furthermore, by Lemma~\ref{lem:fixgood}, $T_{M^{(n)}}(J)\prec T_{\sigma^\ell M^{(n)}}(J)$ for every $1\leq\ell\leq l_n-1$ and so
	$$
	\lambda_n<T_{m_{[k,1]}^{(n)}}(0)\text{ for every }k\notin\{l_n,2l_n,\ldots\}.
	$$
	Hence, $\liminf_{k\to\infty}T_{m_{[k,1]}^{(n)}}(0)=\lambda_n$.
	
	The last claim follows by Theorem~\ref{thm5} and the fact that $\lambda_n$ converges to $\lambda_0$ as $n\to\infty$.
\end{proof}

\begin{proposition}\label{prop:nobetween}
	$(\lambda_1,\infty)\cap\mathcal{L}=\emptyset$, and for every $n\geq1$, $(\lambda_{n+1},\lambda_n)\cap\mathcal{L}=\emptyset$.
\end{proposition}

\begin{proof}
	First, observe that $\max\mathcal{L}=\lambda_1$. Indeed, this follows by Lemma~\ref{lemma0b}\eqref{(vii)} with $K=2$, which implies the first claim.
	
	Let us show that $(\lambda_{n+1},\lambda_n)\cap\mathcal{L}=\emptyset$. Contrary, let us assume that there exists an integer $n\geq1$ and $x\in(\lambda_{n+1},\lambda_n)\cap\mathcal{L}$.
	
	Since $\lambda_n$ is the fixed point of $T_{M^{(n)}}$, we have $\lambda_n\in T_{M^{(n)}M^{(n)}}(J)\subset T_{M^{(n)}}(J)$. Since $\mathcal{L}\subset\Lambda$, where $\Lambda$ is the attractor of $\{T_2,T_3\}$, and $\mathrm{mid}(T_{M^{(n+1)}}(J),T_{M^{(n)}M^{(n)}}(J))\subset(\lambda_{n+1},\lambda_n)$, we get by Lemma~\ref{lem:propM} that either $x\in T_{M^{(n)}M^{(n)}}(J)$ or $x\in T_{M^{(n+1)}}(J)$. But by Lemma~\ref{lem:oddeven}\eqref{it:even}, if $x\in T_{M^{(n)}M^{(n)}}(J)$ then $x\geq\lambda_n$, while if $x\in T_{M^{(n+1)}}(J)$ then $x\leq\lambda_{n+1}$ by Lemma~\ref{lem:oddeven}\eqref{it:odd}, which is a contradiction.
\end{proof}

\begin{proof}[Proof of Theorem~\ref{thm1}]
	The statement follows by combining Proposition~\ref{prop:lambdainL} and Proposition~\ref{prop:nobetween}.
\end{proof}

\section{\bf Proof of Theorem \ref{thm3}}

\begin{proof}	
	It follows from Lemma \ref{lemma0b}\eqref{(vii)} that if $m_n\ge 5$ for infinitely many $n$, then $\liminf _{k\to \infty }T_{m_{[k,1]}}(0)$ $\le \frac{1}{6}$, thus we may assume without loss of generality that $(m_i)\in\{2,3,4\}^{\ZZ^+}$.
	
	If $(m_k,m_{k-1})=(4,2)$ for infinitely many $k$, then  $(m_k,m_{k-1},m_{k-2})=(4,2,a)$ for an $a\in\{2,3,4\}$ and for infinitely many $k$. By Lemma~\ref{lem:oddeven}\eqref{it:odd}
	$$
	\liminf _{k\to \infty }T^{(1)}_{m_{[k,1]}}(0)\le\max\{\mathrm{Fix}(T_4\circ T_2\circ T_a):a\in\{2,3,4\}\}=\frac{3}{17},$$
	so we may assume that $(m_k,m_{k-1})\neq(4,2)$ for every $k$. Thus,
	$$
	\mathcal{L}\cap\left(\frac{3}{17},\frac{1}{3}\right]\subset\{\liminf_{n\to\infty}T_{m_{[n,1]}}(0):m_i\in\{2,3,4\},(m_i,m_{i-1})\neq(4,2)\}=:S.
	$$
	If $m_i\in \{ 2,3,4\}$ and $(m_i,m_{i-1})\ne (4,2)$, then there are 55 possibilities for $(m_i,m_{i-1},m_{i-2},m_{i-3})$:
	\begin{align*}
	A=\{&\text{$(2,2,2,2)$, $(2,2,2,3)$, $(2,2,2,4)$, $(2,2,3,2)$, $(2,2,3,3)$, $(2,2,3,4)$, $(2,2,4,3)$, $(2,2,4,4)$,}\\
	&\text{$(2,3,2,2)$, $(2,3,2,3)$, $(2,3,2,4)$, $(2,3,3,2)$, $(2,3,3,3)$, $(2,3,3,4)$, $(2,3,4,3)$, $(2,3,4,4)$,}\\
	&\text{$(2,4,3,2)$, $(2,4,3,3)$, $(2,4,3,4)$, $(2,4,4,3)$, $(2,4,4,4)$, $(3,2,2,2)$, $(3,2,2,3)$, $(3,2,2,4)$,}\\
	&\text{$(3,2,3,2)$, $(3,2,3,3)$,$(3,2,3,4)$, $(3,2,4,3)$, $(3,2,4,4)$, $(3,3,2,2)$, $(3,3,2,3)$, $(3,3,2,4)$,}\\
	&\text{$(3,3,3,2)$, $(3,3,3,3)$, $(3,3,3,4)$, $(3,3,4,3)$ $(3,3,4,4)$, $(3,4,3,2)$, $(3,4,3,3)$, $(3,4,3,4)$,}\\
	&\text{$(3,4,4,3)$,  $(3,4,4,4)$, $(4,3,2,2)$, $(4,3,2,3)$, $(4,3,2,4)$, $(4,3,3,2)$, $(4,3,3,3)$, $(4,3,3,4)$,}\\
	&\text{$(4,3,4,3)$, $(4,3,4,4)$, $(4,4,3,2)$, $(4,4,3,3)$, $(4,4,3,4)$, $(4,4,4,3)$, $(4,4,4,4)$}\}.
	\end{align*}
	Let us consider the IFS $\Phi=\{T_a\circ T_b\circ T_c\circ T_d|(a,b,c,d)\in A\}$ and let $\Lambda'$ be the attractor of $\Phi$. Then by Lemma~\ref{lem:convergence}, $S\subset\Lambda'$. Moreover, it is easy to check that the sum of the 55 contractions strictly less than 1, therefore by Lemma~\ref{lem:zeromeasure}, $\lambda (\Lambda')=0$, and so $\lambda (\mathcal{L} \cap [\frac{3}{17},\frac{1}{3}])=0$.
\end{proof}


\end{document}